\documentclass[12pt]{amsart}
\usepackage{latexsym}
\usepackage{amsmath}
\usepackage{amsfonts}
\usepackage{amsthm}
\usepackage{amssymb}
\usepackage{ifthen}
\usepackage{enumerate}
\usepackage[T1]{fontenc}
\def\Eq#1#2{\ifthenelse{\equal{#1}{*}}
  {\begin{equation*}\begin{aligned}#2\end{aligned}\end{equation*}}
  {\begin{equation}\begin{aligned}\label{E#1}#2\end{aligned}\end{equation}}}

\newtheorem{thm}{Theorem}
\newtheorem{lem}[]{Lemma}
\newtheorem{prop}[]{Proposition}
\theoremstyle{remark}
\newtheorem{rem}[]{Remark}
\theoremstyle{definition}
\newtheorem{defin}[]{Definition}

\newcommand{\NN}{\mathbb{N}}
\newcommand{\ZZ}{\mathbb{Z}}
\newcommand{\QQ}{\mathbb{Q}}

\newcommand{\Xx}{X}

\newcommand{\dltu}{\mathcal{U}}
\newcommand{\dltj}{\mathbf{J}}

\newcommand{\dltp}{\mathcal{P}}

\newcommand{\dlta}{\mathcal{A}}
\newcommand{\dltb}{\mathcal{B}}
\newcommand{\dltc}{\mathcal{C}}
\newcommand{\dltd}{\mathcal{D}}
\newcommand{\dltt}{\mathcal{T}}
\newcommand{\dltf}{\mathcal{F}}
\newcommand{\dltm}{\mathcal{M}}

\author[Zolt\'an Boros]{Zolt\'an Boros}
\author[P\'eter T\'oth]{P\'eter T\'oth}

\title[Ultrapowers of ordered sets]{Interval 
chains and completeness in 
ultrapowers of ordered sets}

\address[Z.~Boros]{Institute of Mathematics, 
University of Debrecen,
4002 Debrecen, Pf.~400, Hungary}
\email{zboros@science.unideb.hu}

\address[P.~T\'oth]{Institute of Mathematics, 
University of Debrecen,
4002 Debrecen, Pf.~400, Hungary}
\email{peter.toth042@gmail.com}



\thanks{Research of Z.~Boros 
has been supported 
by the K-134191 NKFIH Grant 
and the 2019-2.1.11-T\'ET-2019-00049 project. 
Projects no. 2019-2.1.11-TÉT-2019-00049 
and K134191 have been implemented with the support 
provided from the 
National Research, Development and Innovation Fund 
of Hungary, financed under the 
TÉT and K~20 funding schemes, respectively.\\ 
Research of P.~T\'oth has been supported by
the EFOP-3.6.1-16-2016-00022 project. 
The project was co-financed by the European Union 
and the European Social Fund.} 

\begin{document}

\begin{abstract}
The ultrapower $ T^{\ast} $ 
of an arbitrary ordered set $T$ 
is introduced as an infinitesimal 
extension of $ T \,$. 
It is obtained as the set of 
equivalence classes of the sequences in $ T \,$, 
where the corresponding relation 
is generated by an ultrafilter 
on the set of natural numbers. 
It is established that $T^\ast$ always satisfies 
Cantor's property, while one can give 
the necessary and sufficient conditions for $T$ 
so that $T^\ast$ would be complete or it 
would fulfill the open completeness property, 
respectively.   
Namely, the density of 
the original set determines the 
open completeness of the extension, while independently, 
the completeness of $T^\ast$ is determined 
by the cardinality of $ T \,$.  
 
\end{abstract}

\maketitle

\section{Introduction}

A well known statement from the theory of ordered fields 
is that an ordered field is complete if and only if it 
simultaneously fulfills the Archimedean property and 
Cantor's property. To demonstrate the independence of 
these properties, one needs to construct an ordered field 
which fulfills Cantor's property but is not complete. 
This question is usually treated in the framework of 
non-standard analysis, for instance in the works of 
Stroyan and Luxemburg \cite{SL76}, \cite{Lux73}.  

However in the 
above cited publications 
one can also find 
an idea for a construction that needs only standard tools. 
This idea is the concept of ultrapowers: 
Let us choose an adequate family of subsets 
of the set of natural numbers 
called ultrafilter,  
and then use it 
to define an equivalence relation on the set 
of all sequences of the elements of a given set $R$. 
This provides a partition, and the 
set of the equivalence classes is called the ultrapower of $R$. 
After introducing this concept, the authors proceed using 
mainly non-standard techniques, also when it comes to show
the properties of the ultrapower. 
We note 
that the result of this method is an extension 
of the original set, since the classes of 
the constant sequences can be considered 
as representatives of the original elements. 

The aforementioned works contain only the main ideas 
without the technical parts, 
but in the recent publication \cite{Cor16} of Corazza, 
a detailed construction of the ultrapower of $\QQ$ 
(denoted by $\QQ ^ \NN / U$) is displayed. 
One of the main objectives of his work is 
to construct a non-archimedean 
ordeded field that fulfills Cantor's property 
(in \cite{Cor16} it is referred as Nested Intervals Property). 
As mentioned above, such a construction provides 
an ordered field with the Cantor property which is 
not complete. 

Not surprisingly, only the ordering of the ultrapower 
plays an important role during the investigation of these 
two order-related properties: 
completeness and the Cantor property, 
while the field operations are irrelevant at that point. 
Actually this fact motivates us 
to generalize these constructions, 
introducing the ultrapower of 
an arbitrary ordered set $T$, 
and investigating Cantor's property 
and completeness in its extension $T^\ast$. 

We may note that while ultrafilters play central role 
in all of the constructions cited before, 
the definition of them is 
not completely coherent. 
In fact the ultrafilter has to fulfill 
some conditions which ensure that the extension is proper 
(not trivial), 
but these conditions do not appear 
in the classical definition of 
the ultrafilter, e.g. in the monograph of Jech \cite{Jch03}. 
Hence in \cite{Lux73} and \cite{SL76} a so-called 
free ultrafilter is used, 
while in \cite{Cor16} a nonprincipal ultrafilter is used. 
The main inconvenience with these 
special ultrafilters is to prove their existence -- 
typically it is done by using Zorn's lemma.  
Therefore we find it useful to revisit this question, 
and define the concept of an ultrafilter in such a way 
that it would be suitable for our 
construction, and its existence would follow 
relatively easily from Tarski's classical 
existence theorem for 'ordinary' ultrafilters 
(which may be found in \cite[Theorem 7.5]{Jch03}). 

Our process of showing Cantor's property for the ultrapower 
sometimes resembles Corazza's methods, although at one point 
the fact that we start from an arbitrary ordered set makes a 
significant difference. 
Namely, in \cite{Cor16} it is shown that in 
$\QQ ^ \NN / U$ the 
intersection of a chain of countably many open intervals is 
nonempty. This property is referred to as open completeness, 
and it clearly implies both Cantor's property and 
the lack of completeness. 
In our more abstract setting it is reasonable 
to investigate these properties separately.  
Namely, starting from an ordered set $ T \,$, 
we shall prove that $T^\ast$ always satisfies 
Cantor's property, while we can give 
the necessary and sufficient conditions for $T$ 
so that $T^\ast$ would be complete or it 
would fulfill the open completeness property, 
respectively.   
Namely, the density of 
the original set determines the 
open completeness of the extension, while independently, 
the completeness of $T^\ast$ is determined 
by the cardinality of $ T \,$.

\section{Particular properties of ordered sets}

In this section we collect the basic concepts 
for ordered sets that are in the focus of this paper. 

As usual, we call a nonempty set $\Xx$ 
equipped with a relation $ \leq $ (on $\Xx$) 
an ordered set if the relation $ \leq $ 
is reflexive, anti-symmetric, transitive, 
and linear (i.e., $ x \leq y $ or $ y \leq x $ 
for all $ x,y \in X $). 

Once the relation $ \leq $ on $\Xx$ is given, we 
shall also use the relations 
$ \geq \,$, $ < $ and $ > $ in the usual sense. 

We shall use the concepts of {\em lower/upper bound}, 
{\em minimum/maximum} 
(denoted by $ \mathrm{min} $ and $ \mathrm{max} \,$, 
respectively), 
a set being {\em bounded from below/above}, 
{\em least upper bound} ($\mathrm{sup}$) and 
{\em greatest lower bound} ($\mathrm{inf}$) 
in the usual sense as well 
(cf. \cite[Definition 2.2]{Jch03}). 

\begin{defin}
An ordered set $(\Xx, \leq)$ is called {\em complete} 
if every nonempty subset of $\Xx$, that is bounded from above, 
has a least upper bound.
\end{defin}

It is well known that the ordered set $(\Xx, \leq)$ 
is complete if, and only if, 
every nonempty subset of $\Xx$, 
that is bounded from below, 
has a greatest lower bound 
(a proof in an abstract setting can be found, 
for instance, in \cite[Theorem 4.6]{BSz08}). 

We will define {\em intervals} 
as particular subsets of an ordered set 
$(\Xx, \leq)$ in the usual way. 
For example, if $ a,b \in X $ such that $ a < b \,$, 
let  
$ [a,b[ = \{\, x \in X \,:\,  a \leq x < b \,\} $.  

We call a sequence $(I_n)$ of non-empty intervals an 
{\em interval chain} if $ I_{n+1} \subset I_n $ 
for every $ n \in \NN \,$.  
We can describe Cantor's property and 
the open completeness of an ordered set $X$ 
by the phenomena that the intersection of an arbitrary 
interval chain of closed, respectively, open intervals 
is non-empty.  

\begin{defin} \label{DCantor}  
We say that an ordered set $X$ satisfies 
{\em Cantor's property} if 
\[
\bigcap_{n \in \NN} [a_n , b_n] \neq \emptyset 
\]
for any sequences $ (a_n),(b_n) : \NN \rightarrow X $ 
fulfilling 
\[ 
a_k \leq a_{k+1} \leq b_{k+1} \leq b_k   
\] 
for every $ k \in \NN \,$. 
\end{defin}

\begin{defin} \label{Dopencomplete}  
We say that an ordered set $X$ is {\em open complete} if 
\[
\bigcap_{n \in \NN} \, ]a_n , b_n[ \, \neq \, \emptyset 
\]
for any sequences $ (a_n),(b_n) : \NN \rightarrow X $ 
fulfilling 
\[ 
a_k \leq a_{k+1} < b_{k+1} \leq b_k   
\] 
for every $ k \in \NN \,$. 
\end{defin} 

Finally, we introduce the concept of density in ordered sets. 

\begin{defin} \label{Ddensity}
We say that an ordered set $X$ is {\em dense everywhere} 
if, for any $a,b \in X$ fulfilling $a < b$, 
there exists $c \in X$ such that $a < c <b$. 
\end{defin}

\section{An extension of ordered sets}

\subsection{Ultrafilter}

We introduce the concept of ultrafilter. 
For the power set of an arbitrary set $\Xx$ we will use the 
notation $\mathcal{P}(\Xx)$, i.e. the elements of 
$\mathcal{P}(\Xx)$ are the subsets of $\Xx$. 

\begin{defin} \label{ultdef}
Let $\dltj$ be an infinite set. 
The nonempty family of sets 
$\mathcal{U} \subset \mathcal{P}(\dltj)$ 
is called a {\em filter} on $\dltj$, if
\begin{itemize}
\item[(1)] 
$K \in \mathcal{U}$ and $K \subset L\subset \dltj$ 
implies $L\in \dltu $, 
\item[(2)] 
$K,L \in \dltu $ implies $K\cap L \in \dltu$,
\item[(3)] 
$K \in \dltu$ implies that $K$ is infinite. 
\end{itemize}
Moreover, 
$\dltu$ is called an {\em ultrafilter} 
if it is a filter 
and 
\begin{itemize}
\item[(4)] if $K \subset \dltj$, then 
$K \in \dltu$ or $\dltj \setminus K \in \dltu$ holds.
\end{itemize}
\end{defin}

\begin{rem} In many works (such as \cite{Lux73} or 
\cite{Jch03}) 
filters and ultrafilters are defined on arbitrary 
sets and not particularly infinite ones. In that general case 
assumption {\em(3)} is replaced by 
the weaker condition 
\begin{itemize}
\item[(3')]$\emptyset \notin \dltu$. 
\end{itemize}
In that weaker sense it holds that any filter can be extended 
to an ultrafilter (see \cite[Theorem 7.5]{Jch03}). 
Using this result we prove the following statement. 

\begin{thm}\label{ultfiltthm}
Let $\dltj$ be an infinite set and let 
$K \subset \dltj$ be also infinite. Then there exists an 
ultrafilter $\dltu \subset \dltp (\dltj)$ such that 
$K \in \dltu$. 
\end{thm}

\begin{proof}
Let us define the so-called Fréchet-filter: 
\[
\dltf = \{ S \subset \dltj \ \vert \ \dltj \setminus S 
\mbox{ is finite.} \}
\]
It is easy to see that $\dltf$ is indeed a filter. Let us 
define another subset of $\dltp (\dltj)$: 
\[
\dltm = \{ M \subset \dltj \ \vert \ \exists L \in \dltf 
\ : \ K \cap L \subset M \}.
\]
Now we show that $\dltm$ is a filter (in the weaker sense). 
Let $M,N$ be arbitrary sets in $\dltm$, thus there exist 
$L_M \, ,  L_N \in \dltf$ such that 
$K \cap L_M \subset M$ and 
$K \cap L_N \subset N$. 

\begin{itemize}
\item[(1)] If $M \subset S$ then $K \cap L_M \subset M \subset S$, 
so $S \in \dltm$. 

\item[(2)] $K \cap (L_M \cap L_N) = 
(K \cap L_M) \cap (K \cap L_N) \subset M \cap N$, 
and as $L_M \cap L_N \in \dltf$, it also holds that
$M \cap N \in \dltm$.

\item[(3')] Assume $\emptyset \in \dltm$, which means 
$K \cap L = \emptyset$ for some $L \in \dltf$. But this 
would imply $K \subset \dltj \setminus L$ and that is impossible, 
since $K$ is infinite and $\dltj \setminus L$ is finite. 
\end{itemize}

Notice that $\dltf \subset \dltm$ trivially holds. Now 
$\dltm$ can be extended to an ultrafilter $\dltu$ 
(again in the weaker sense). However $\dltu$ is an ultrafilter 
in our restrictive sense, too. 
Indeed, if $F \in \dltu$ for some 
finite subset $F$, then $\dltj \setminus F \notin \dltu$ 
which contradicts $\dltj \setminus F \in \dltf$. 
\end{proof}

As a corollary of this statement, we get that there
exists an ultrafilter on the set of natural numbers. 
Finally we emphasize that the existence of an appropriate 
ultrafilter on $\NN$ can be proven in several, 
slightly different ways. 
However this often requires the introduction of further 
definitions such as {\em free ultrafilter} (in \cite{Lux73}) 
or {\em nonprincipal ultrafilter} (in \cite{Cor16}). 
\end{rem} 

\subsection{Ultrapower of an ordered set}

In the next step we construct a so called 
{\em ultrapower} of any ordered set $T$. 
The existence of an ultrafilter on the set of natural numbers
provides us a way to define an equivalence relation 
on the set of all sequences of elements of $T$, 
in such manner that an adequate order 
on the equivalence classes 
would generate an ordered set. 
As it is common in the literature, 
we will use an asterisk ($^\ast$) to denote 
the operation that assigns its ultrapower to the 
original ordered set. 

In the subsequent sections let $T$ be an ordered set and 
$\dltu$ be an ultrafilter on $\NN$.

Let 
$\dltt = 
\{(a_n) \ \vert \ (a_n) : \mathbb{N} \rightarrow T \} $ 
denote the set of all sequences of elements of $T$. 

\begin{prop}
Let us define the relation 
$\sim \ \subset \dltt \times \dltt$ in the 
following way: 
\[
(a_n) \sim (b_n) \Longleftrightarrow 
\{ n \in \mathbb{N} : a_n = b_n \} \in \dltu.
\]
Then $\sim$ is an equivalence relation. 
Furthermore, let us 
denote the set of the equivalence classes by $T^\ast$, 
while the class of an element $(a_n) \in \dltt$ 
be denoted by $ \overline{(a_n)} \,$. 
The relation $\leq \ \subset T^\ast \times T^\ast$, given by 
\[
\overline{(a_n)} \leq \overline{(b_n)} \Longleftrightarrow 
\{ n \in \mathbb{N} : a_n \leq b_n \}\in \dltu \,, 
\]
is well-defined, and $(T^\ast, \leq)$ is an ordered set. 
\end{prop}

\begin{proof}
The reflexivity and symmetry of $\sim$ is obvious. 
To check the transitivity, assume $(a_n) \sim (b_n)$ 
and $(b_n) \sim (c_n)$. Then 
\[
\{ n \in \NN : a_n = c_n \} \supset
\{ n \in \NN : a_n = b_n \} \cap 
\{ n \in \NN : b_n = c_n \} \in \dltu 
\hspace{2cm} (\star)
\]
implies $(a_n) \sim (c_n)$, so $\sim$ is indeed an 
equivalence relation. Similarly, if 
$(a_n) \sim (\tilde{a}_n) \, $, $(b_n) \sim (\tilde{b}_n)$ 
and $\overline{(a_n)} \leq \overline{(b_n)}$, then 
\begin{align*}
\{n \in \NN : \tilde{a}_n \leq \tilde{b}_n \} \supset  
\{n \in \NN : a_n \leq b_n \} \cap 
\{n \in \NN : a_n = \tilde{a}_n \} \cap 
\{n \in \NN : b_n = \tilde{b}_n \} \in \dltu 
\end{align*} 
ensures that 
$\overline{(\tilde{a}_n)} \leq \overline{(\tilde{b}_n)}$, 
hence $\leq$ is independent of the choice of 
representatives, i.e. it is a well-defined relation on 
$T^\ast$.

Clearly $\leq$ is reflexive, and also notice that if we 
replace the equalities with inequalities in $(\star)$, 
we get the transitivity of $\leq$. Furthermore, 
\[
\{ n \in \NN : a_n = b_n \} \supset
\{ n \in \NN : a_n \leq b_n \} \cap 
\{ n \in \NN : b_n \leq a_n \} \in \dltu
\]
shows that $\leq$ is antisymmetric. Finally, since the sets 
$\{n \in \NN : a_n \leq b_n \}$ and $\{ a_n > b_n \}$ give 
a disjoint partition of $\NN$, exactly one of them is in 
$\dltu$. These properties together provide that 
$(T^\ast, \leq)$ is an ordered set. 
\end{proof}


\subsection{Cantor's property for the extension}

In this section we will show that 
the operation $ ^\ast $ always produces 
an ordered set that satisfies Cantor's property. 

\begin{thm} \label{mainthrm}
If $T$ is an ordered set then its extension $ T^{\ast} $ satisfies Cantor's property, i.e. if 
$ a_k = ((a_k)_n) \in \dltt $ and $ b_k = ((b_k)_n) \in \dltt $ $(k \in \NN)$ 
such that for every $ k \in \NN $ 
\[ 
\overline{a_k} \leq \overline{a_{k+1}} \leq \overline{b_{k+1}} \leq \overline{b_k} \,, 
\] 
then
\[ 
\bigcap_{k \in \NN} \left[ \overline{a_k} \,,\, \overline{b_k} \right] 
\neq \emptyset \,.
\]
\end{thm}

\begin{proof}

We define the following sets:

\begin{gather*}
A_i = \{n \in \NN : (a_i)_n \leq (a_{i+1})_n \}, \hspace{1cm}
B_i = \{n \in \NN : (b_i)_n \geq (b_{i+1})_n \}, \\ 
\mbox{ and } C_i=\{ n \in \NN : (a_i)_n \leq (b_i)_n\} 
\hspace{5mm} \mbox{for every } i \in \NN.
\end{gather*}
 
Using these we construct the following sets: 

\[
\dlta_k=\bigcap_{i=1}^{k-1} A_i, \mbox{\hspace{10mm} } 
\dltb_k=\bigcap_{i=1}^{k-1} B_i, \mbox{\hspace{10mm} } 
\dltc_k=\bigcap_{i=1}^{k} C_i \mbox{\hspace{10mm}} 
(k\in\mathbb{N} \setminus \{ 1 \} ).
\]

Obviously $\dlta_k$, $\dltb_k$, $\dltc_k$ belong to $\dltu$, as they are intersections of finitely many sets from $\dltu$. 
For the same reason 
$ \dlta_k \cap \dltb_k \cap \dltc_k = \dltd_k \in \dltu $. 

Let $ \dltd_1 = \dltc_1 = C_1 \,$, 
thus the set $\dltd_k$ is now defined 
for every $k \in \NN$, and it consists of 
the natural numbers $n$ for which 
the following inequalities hold: 
\[
(a_1)_n \leq \hdots \leq (a_k)_n \leq 
(b_k)_n \leq \hdots \leq (b_1)_n.
\]

In the next step, for every $ n \in \NN \,$, 
we define another 
set of natural numbers $I_n$ as follows: 
$I_n=\{k\in\mathbb{N}: n\in \dltd_k\}$.
It is easy to see from the definition of the sets $\dltd_k$ 
that if $k \in I_n$ then $l \in I_n$ is 
also true for every natural number $l \leq k$.
Using these sets we assign a non-negative integer 
to every  $ n \in \NN $ as follows: let 
\[
\alpha_n=
\begin{cases}
0, & \text{if } I_n = \emptyset \,,\\
n, &  \text{if } I_n \text{ has no upper bound},\\
\min \{ n, \max I_n \} &  \text{if } I_n 
\text{ is non-empty and bounded from above.}
\end{cases}
\] 
We should note that if $n$ is an element of $\dltd_k$ 
and $k \leq n$ then 
$ \alpha_n \geq k $ ($k,n \in \NN$ ). 
This also means that $n \notin \dltd_1$ holds 
if and only if $I_n = \emptyset$.
Hence 
$n \in \dltd_k \setminus \{ m\in  \NN : m < k \}$  
implies 
\[ 
(a_1)_n\leq \hdots \leq (a_k)_n \leq (a_{\alpha_n})_n \leq
(b_{\alpha_n})_n \leq (b_k)_n \leq \hdots \leq (b_1)_n.
\]
It is trivial that 
$ \dltd_k \setminus \{ m\in  \NN : m < k \}$ 
is in the ultrafilter.

After these remarks it is rather easy to construct a 
common point of the interval chain. We define the sequence 
$c = (c_n) : \mathbb{N}\longrightarrow\mathbb{R}$ as follows:

\[
c_n = 
\begin{cases}
(a_{\alpha_n})_n, & \text{if } n \in \dltd_1\\
(a_1)_1, & \text{if } n \notin \dltd_1.
\end{cases}
\]

We will show that 
$ \overline{a_k} \leq \overline{c} \leq \overline{b_k} $ 
for any $k \in \NN$. 
This is a straightforward
corollary of our previous remark, namely that 
\[
(a_1)_n\leq \hdots \leq (a_k)_n \leq (a_{\alpha_n})_n 
= c_n \leq (b_k)_n \leq \hdots \leq (b_1)_n 
\]
holds if 
$n \in \dltd_k \setminus \{ m\in  \NN : m < k \}$. 
Since 
$\dltd_k \setminus \{ m\in  \NN : m < k \} \in \dltu$, 
the sets 
\[
 \{ n \in \NN : (a_k)_n  \leq c_n \} \ \mbox{ and } \ 
\{ n \in \NN : c_n \leq (b_k)_n \} 
\]
are also elements of $\dltu$ 
(obviously they are supersets of 
$\dltd_k \setminus \{ m\in  \NN : m < k \} $). 

The final step of the proof is to use the definition of the 
ordering relation on $T^\ast$, so we obtain 
\[
\overline{c} \in 
\bigcap_{k\in\mathbb{N}} [\overline{a_k},\overline{b_k}]
\]
\end{proof}

\begin{rem} \label{Rem01}
It seems reasonable to make a similar proposition 
and replace the closed intervals by open intervals. 
However if we do so then we must 
require the $T$ ordered set to be dense everywhere. 
Otherwise a trivial counterexample can be made 
as an empty open interval exists.

On the other hand, the criterion concerning the density of $T$ 
is sufficient to prove the alternate form of 
the previous theorem (i.e. open completeness). 
We sum up these perceptions in the following theorem: 
\end{rem}

\begin{thm}
Let $T$ be an ordered set. 
The following statements are equivalent:
\begin{itemize}
\item[(a)] $T$ is dense everywhere.
\item[(b)] 
if 
$ a_k = ((a_k)_n) \in \dltt $ and 
$ b_k = ((b_k)_n) \in \dltt $ $(k \in \NN)$ 
such that for every $ k \in \NN $ 
\[ 
\overline{a_k} \leq \overline{a_{k+1}} < 
\overline{b_{k+1}} \leq \overline{b_k} \,, 
\] 
then
\[ 
\bigcap_{k \in \NN} 
\left] \overline{a_k} \,,\, \overline{b_k} \right[ 
\neq \emptyset \,.
\]
\end{itemize}
\end{thm}

\begin{proof}
To show (b) $\Longrightarrow$ (a) we explain the 
counterexample which was mentioned in Remark \ref{Rem01}. 
Let $p,q \in T$ such that
$p < q $ and there is no element of $T$ in the open interval $]p,q[$.
 
This means that the open interval 
$ \left] \overline{p},\overline{q} \right[ $ 
is also empty, where 
$\overline{p}$ and $\overline{q}$ are the classes of the 
constant sequences $(p_n)$ and $(q_n)$ 
defined by $ p_n = p $ and $ q_n = q $ 
for every $ n \in \NN \,$. 

Thus if 
$ p_k = (p_n) \in \dltt $ and $ q_k = (q_n) \in \dltt $ 
for every $k \in \NN$, then 
\[
\bigcap_{k \in \NN} 
\left] \overline{p_k} \,,\, \overline{q_k} \right[ 
= \left] \overline{p} , \overline{q} \right[ 
= \emptyset \,. 
\] 

To show the reverse implication, we can take the 
same process as we did in the proof of Theorem \ref{mainthrm}.
The only adjustments to be made are that 
we define 
\[ 
C_i = \{\, n \in \NN \,:\, (a_i)_n < (b_i)_n \,\} 
\quad (i \in \NN) 
\]
and 
$c_n$ has to be an element from the interior of 
the '$\alpha_n$-th' interval 
(obviously, it cannot be an endpoint as it initially was),  
that is, 
\[ 
c_n \in 
\left] (a_{\alpha_n})_n \,,\, (b_{\alpha_n})_n \right[ 
\ \ \mbox{if} \ \ n \in \dltd_1 \,, 
\ \ \mbox{while} \ \ 
c_n = (a_1)_1 \ \ \text{if} \ \ n \notin \dltd_1 \,.
\]  
Clearly, the required element $c_n$ exists 
as $T$ is dense everywhere.
We will not repeat the entire proof 
since every remaining step is analogous. 
\end{proof}

\subsection{Completeness of the extension}

Finally we will show that the operation $ { }^{\ast} $ 
does not preserve completeness in general. 
Moreover the completeness of the ultrapower  
depends only on the cardinality of the initial ordered set.

\begin{lem} \label{Lemma1}
Let $\dltu$ be an ultrafilter on $\NN$ 
and $ A_j \subset \NN $ $ \ (j = 1,\dots,n) $.  
If 
\[
\bigcup_{j=1}^{n} A_j = \NN \, \text{, then } 
\exists k \in \{ 1, \dots ,n \} \text{ such that } 
A_k \in \dltu
\]
\end{lem}

\begin{proof}
Assume that for all indices 
$ j \in \{ 1, \dots ,n \} : A_j \notin \dltu$. 
From the definition of $\dltu$ we get 
$\NN \setminus A_j \in \dltu $ 
for every $j \in \{ 1, \dots ,n \}$. 
This means 
\[
\emptyset =\NN \setminus \NN = 
\NN \setminus \bigcup_{j=1}^{n} A_j = 
\bigcap_{j=1}^{n} \left( \NN \setminus A_j \right) \in \dltu
\]
which is an obvious contradiction as $\emptyset$ is not infinite. 
Therefore some $k \in \{ 1, \dots , n \}$ must exist 
for which $A_k \in \dltu $.
\end{proof}

\begin{thm}
Let $T$ be an ordered set. $T ^\ast$ 
is complete if and only if $T$ 
is finite.
\end{thm}

\begin{proof}
In the first place we prove that if $T$ is infinite then 
$T ^\ast$ is not complete.
We will use the following basic fact:
in an infinite ordered set there exists 
a strictly monotone sequence of elements. 
In order to prove this, we may consider 
an obviously existing injective sequence 
$ (x_n) \colon \NN \rightarrow T $ 
(i.e., $ x_n \neq x_m $ if $ n \neq m $). 
It is a well-known fact 
that every sequence in an ordered set 
contains a monotone subsequence 
(we can apply the proof for real sequences \cite{NP88} 
in this more general context as well). 
Clearly, such a monotone subsequence of $(x_n)$ 
is strictly monotone. 

We give the details of the proof only 
for the case of a strictly increasing sequence.

Let $ t_1 < t_2 < t_3 < \dots $ 
be a strictly increasing sequence of 
elements in $T$. 
It is easy to see that the equivalence classes of 
the constant sequences  
\[ 
(s_k)_n = t_k \qquad (n \in \NN) \ (k \in \NN) 
\] 
generate a subset 
\[ 
S = \{\, \overline{s_k} \,\mid\, k \in \NN \,\} 
\] 
of $T ^\ast$ 
which is bounded from above. 
Indeed, one can easily check that 
$ \overline{(t_n)} $ is an upper bound of $ S \,$.
Now we demonstrate that $S$ 
has no least upper bound. 
Let $ (b_n) \in \dltt $ such that 
$ \overline{(b_n)} $ is an upper bound of $ S \,$. 
We define some sets in 
a similar manner as we did in the proof of 
Theorem \ref{mainthrm}: let 
\[
\dltd_k = \{ n \in \NN \ : \ b_n \geq t_k \} \in \dltu, 
\qquad
I_n = \{ k \in \NN \ : \ b_n \geq t_k \} 
\qquad (k,n \in \NN).
\]
We should note that if $k \in I_n$ then 
$l \in I_n$ for every natural number $l \leq k$. 
Another easy observation is that,  
for any $ m,k \in \NN \,$, 
$ m \in \dltd_k $ if and only if 
$ k \in I_m \,$.  
(we will use these two remarks later on).

Now we can define a mapping 
$ \alpha : \NN \rightarrow \NN \cup \{ 0 \} $  
as follows: let 
\[
\alpha_n =
\begin{cases}
0, & \text{if } I_n = \emptyset \,, \\
n, &  \text{if } I_n \text{ has no upper bound}, \\
\min \{ n, \max I_n \}, &  \text{otherwise}.
\end{cases}
\]
With the notation 
$ \beta_n = \left\lfloor \displaystyle 
\frac{\alpha_n}{2} \right\rfloor $ 
it is possible to construct 
an upper bound for $S$ which is smaller than $\overline{b}$ 
(here $\lfloor \ \rfloor$ denotes the floor, i.e., 
$ \lfloor x \rfloor = 
\max \{\, z \in \ZZ \,:\, z \leq x \,\} $).

We define 
$(c_n) \in \dltt $ as follows:
\[
c_n=
\begin{cases}
b_n, & \text{if } \alpha_n < 2 \\
t_{\beta_n}, &  \text{if } \alpha_n \geq 2
\end{cases}
\]
For any natural number $k$ 
the following argumentation can be made: 
if $ \alpha_n \geq 2k $ then $\beta_n \geq k $ 
and therefore $c_n \geq t_k$.
Since 
\[
\{n \in \NN \ : \ c_n \geq t_k \} \supset  
\{n \in \NN \ : \ \alpha_n \geq 2k \} = 
\dltd_{2k} \setminus \{m \in \NN \ : \ m < 2k \} \in \dltu
\]
follows from the two simple remarks that were stated earlier, 
we have obtained that $\overline{(c_n)}$ is 
an upper bound of $S$. 
On the other hand, for every 
$ m \in \dltd_2 \setminus \{ 1 \} \,$, 
the value $c_m$ is indeed smaller than $b_m$, because 
$ 2 \leq {\alpha}_m \in I_m \,$, and thus 
\[
c_m = t_{\beta_m} < t_{\alpha_m} \leq b_m,
\]
so $\overline{(c_n)} < \overline{(b_n)}$. 
Therefore $S$ has no least upper bound. 

With some obvious adjustments it can be shown that if 
$ u_1 > u_2 > u_3 > \dots $ 
is a strictly decreasing sequence of elements in $T$ 
and $ v_k \in \dltt $ such that $ (v_k)_n = u_k $ 
for all $ n,k \in \NN \,$,  
then the set  
\[
V = \{ \ \overline{v_k} \,\mid\, k \in \NN \,\} 
\subset T^\ast
\]
does not have a greatest lower bound.

In the second part of the proof 
we will verify the reverse implication, 
namely that if $T$ is finite then $T ^\ast$ is complete. 
Since a finite ordered set is always complete, 
it is sufficient to show that, 
for any finite ordered set $ T \,$, 
$ T^{\ast} $ is finite as well. 

Let $ k \in \NN \,$, 
$ T = \{\, t_1 \,,\, \dots \,,\, t_k \,\} $, 
and for each $ j \in \{ 1,\dots,k \} $, 
let $ s_j \in \dltt $ such that 
$ (s_j)_n = t_j $ for all $ n \in \NN $  
(a constant sequence). 
Now let us consider an arbitrary sequence 
$ (a_n) \in \dltt \,$. 
For every $ j \in \{ 1,\dots,k \} $ 
we define the sets 
$ A_j = \{\, n \in \NN \,:\, a_n = t_j \,\} $. 
Obviously,  $ \bigcup_{j=1}^{k} A_j = \NN \,$. 
According to Lemma \ref{Lemma1},  
there exists an index 
$ m \in \{ 1,\dots,k \} $ such that 
$ A_m \in \dltu $ and therefore 
$ \overline{(a_n)} = \overline{s_m} $ 
(i.e., $(a_n)$ is the equivalent with 
the constant $ t_m $ sequence). 
So we may conclude that $T ^\ast$ contains only the
equivalence classes of finitely many constant sequences, 
which implies that $ T^{\ast}$ is complete as well.
\end{proof}


\begin{thebibliography}{20}

\bibitem{BSz08} 
Z.~Boros, \'{A}.~Sz\'{a}z, 
{\em Infimum and supremum completeness properties 
of ordered sets without axioms}, 
An. Ştiinţ. Univ. "Ovidius'' Constanţa Ser. Mat. 
\textbf{16}/2 (2008), 31--37.  

\bibitem{Cor16}
P.~Corazza.
{\em Revisiting the Construction of the Real Line},
International Mathematical Forum 
\textbf{11}/2 (2016), 71--94. 
http://dx.doi.org/10.12988/imf.2016.512104

\bibitem{Jch03}
T.~Jech,
{\em Set Theory},
Springer-Verlag, Berlin -- Heidelberg, 2003. 

\bibitem{Lux73} 
W.~A.~Luxemburg, 
{\em What Is Nonstandard Analysis?}, 
Amer. Math. Monthly 
\textbf{80}/6 (1973), 38--67. 

\bibitem{NP88}
D.~J.~Newman, T.~D.~Parsons, 
{\em On monotone subsequences}, 
Amer. Math. Monthly 
\textbf{95}/1 (1988), 44--45.

\bibitem{SL76} 
K.~D.~Stroyan, W.~A.~J.~Luxemburg, 
{\em Introduction to the theory of infinitesimals}, 
Academic Press, New York -- San Francisco -- London, 1976.

\end{thebibliography}
\end{document}